\documentclass[a4paper,12pt,dvipdfmx]{amsart}
\usepackage{amsmath}
\usepackage{amssymb}
\usepackage{enumerate}
\usepackage{amscd}
\usepackage{graphics}
\usepackage{latexsym}
\usepackage{verbatim} 
\usepackage{url}

\theoremstyle{plain}
\newtheorem{thm}{{\bf Theorem}}[section]
\newtheorem{cor}[thm]{{\bf  Corollary}}
\newtheorem{prop}[thm]{{\bf Proposition}}
\newtheorem{lemma}[thm]{{\bf Lemma}}

\newtheorem{claim}[thm]{{\bf Claim}}

\theoremstyle{definition}
\newtheorem{define}[thm]{{\bf Definition}}
\newtheorem{note}[thm]{{\bf Note}}

\newcommand{\cf}{\mathord{\mathrm{cf}}}

\newcommand{\dom}{\mathord{\mathrm{dom}}}
\newcommand{\size}[1]{\left\vert {#1} \right\vert}
\newcommand{\p}{\mathcal{P}}
\newcommand{\ot}{\mathord{\mathrm{ot}}}

\newcommand{\col}{\mathord{\mathrm{Col}}}

\newcommand{\seq}[1]{\langle {#1} \rangle}

\newcommand{\norm}[1]{\left\| {#1} \right\|}

\newcommand{\ka}{\kappa}
\newcommand{\la}{\lambda}
\newcommand{\om}{\omega}
\newcommand{\ga}{\gamma}

\newcommand{\calF}{\mathcal{F}}
\newcommand{\calG}{\mathcal{G}}

\newcommand{\calM}{\mathcal{M}}
\newcommand{\calN}{\mathcal{N}}

\newcommand{\HOD}{\mathrm{HOD}}

\newcommand{\ZF}{\mathsf{ZF}}
\newcommand{\ZFC}{\mathsf{ZFC}}
\newcommand{\AC}{\mathsf{AC}}
\newcommand{\DC}{\mathsf{DC}}

\title[A note on L\"owenheim-Skolem cardinals]{A note on L\"owenheim-Skolem cardinals}
\author[T. Usuba]{Toshimichi Usuba}
\date{\today}
\address[T. Usuba]
{Faculty of Science and Engineering,
Waseda University, 
Okubo 3-4-1, Shinjyuku, Tokyo, 169-8555 Japan}
\email{usuba@waseda.jp}

\begin{document}
\begin{abstract}
In this note we provide some applications of
L\"owenheim-Skolem cardinals introduced in \cite{U}.
\end{abstract}

\maketitle
\section{Introduction}
Throughout this note, our base theory is $\ZF$ 
unless otherwise specified. 
In Usuba \cite{U}, we introduced the notion of \emph{L\"owenheim-Skolem cardinal},
which corresponds to the downward L\"owenheim-Skolem theorem in the context of $\ZFC$.

\begin{define}
Let $\ka$ be an uncountable cardinal.
\begin{enumerate}
\item $\ka$ is \emph{weakly L\"owenheim-Skolem} (weakly LS, for short)
if for every $\gamma<\ka$, $\alpha \ge \ka$, and $x \in V_\alpha$,
there is $X \prec V_\alpha$ such that $V_\gamma \subseteq X$,
$x \in X$, and the transitive collapse of $X$ belongs to $V_\ka$.
\item $\ka$ is \emph{L\"owenheim-Skolem} (LS, for short)
if for every $\gamma<\ka$, $\alpha \ge \ka$, and $x \in V_\alpha$,
there is $\beta \ge \alpha$ and $X \prec V_\beta$ such that $V_\gamma \subseteq X$,
$x \in X$, ${}^{V_\gamma} (X \cap V_\alpha) \subseteq X$, 
and the transitive collapse of $X$ belongs to $V_\ka$.
\end{enumerate}
\end{define}
Note that if $\ka$ is a limit of LS (weakly LS, respectively) cardinals,
then $\ka$ is LS (weakly LS, respectively).
In \cite{U}, we proved that some non-trivial consequences of the Axiom of Choice $\AC$ holds
on the successor of a singular weakly LS cardinal:

\begin{thm}\label{1.2}
Suppose $\ka$ is a singular weakly LS cardinal (e.g., a singular limit of
weakly LS cardinals).
\begin{enumerate}
\item There is no cofinal map from $V_\ka$ into $\ka^+$,
hence $\ka^+$ is regular.
\item For every function $f$ from $V_\ka$ into the club filter over $\ka^+$,
the intersection $\bigcap f``V_\ka$ contains a club in $\ka^+$.
Hence the club filter is $\ka^+$-complete.
\item For every regressive function $f:\ka^+ \to \ka^+$,
there is $\alpha<\ka^+$ such that
the set $\{\eta<\ka^+ \mid f(\eta)=\alpha\}$ is stationary.
\end{enumerate}
\end{thm}
See \cite{U} for more information,
 and applications to choiceless set-theoretic geology.

In this note, we provide related results
and other applications of LS cardinals,
and we show that various consequences of $\AC$ follows from LS cardinals.

We will use the following notion in this note.
For a set $x$, let $\norm{x}$ be the least ordinal $\alpha$
such that there is a surjection from $V_\alpha$ onto $x$.
Note that:
\begin{itemize}
\item $\norm{V_\alpha}=\alpha$ for every infinite ordinal $\alpha$.
\item If $x \subseteq y$, or there is a surjection from $y$ onto $x$, then $\norm{x} \le \norm{y}$.
\item For every ordinal $\alpha$, there is a cardinal $\ka$ with
$\norm{\ka}>\alpha$.
\item If $\ka$ is weakly LS, then for every $\gamma<\ka$
there is no surjection from $V_\gamma$ onto $\ka$ (\cite{U}),
so $\norm{\ka}=\ka$.
\item Suppose $\ka=\norm{\ka}$. Let $\alpha \ge \ka$ 
and $X \prec V_\alpha$ be with $\norm{X}<\ka$.
Then the transitive collapse of $X$ is in $V_\ka$;
Let $\overline{X}$ be the transitive collapse of $X$,
and $\delta$ the rank of $\overline{X}$.
There is a surjection from $\overline{X}$ onto $\delta$,
so $\norm{\delta} \le \norm{\overline{X}}$.
Since the collapsing map is a bijection from $X$ onto $\overline{X}$, 
we have $\norm{\delta} \le \norm{\overline{X}} =\norm{X}<\ka$.
If $\delta \ge \ka$, then we know $\norm{\delta} \ge \norm{\ka}=\ka$, but this is impossible
and we obtain $\delta<\ka$.
\end{itemize}
Using this notion,
we can reformulate LS and weakly LS cardinals as follows:
\begin{itemize}
\item $\ka$ is weakly LS if and only if $\norm{\ka}=\ka$, and for every $\alpha \ge \ka$ and set $A \subseteq V_\alpha$ with $\norm{A}<\ka$,
there is $X \prec V_\alpha$ such that  $A \subseteq X$ and $\norm{X}<\ka$.
\item $\ka$ is LS if and only if $\norm{\ka}=\ka$, and for every $\gamma<\ka$, $\alpha \ge \ka$ and set $A \subseteq V_\alpha$ with $\norm{A}<\ka$,
there is $\beta \ge \alpha$ and $X \prec V_\beta$ such that  
$A \subseteq X$, $\norm{X}<\ka$, and for every $y \subseteq X \cap V_\alpha$,
if $\norm{y} \le \gamma$ then $y \in X$.
\end{itemize}

\section{On elementary submodels}
First let us present somewhat trivial characterization.
\begin{prop}
Let $\ka$ be an uncountable cardinal.
Then the following are equivalent:
\begin{enumerate}
\item $\ka$ is weakly LS.
\item $\norm{\ka}=\ka$ and for every first order structure $\calM=\seq{M;\ldots}$ with countable language 
and $A \subseteq M$ with $\norm{A} <\ka$,
there is an elementary submodel $\calN=\seq{N; \ldots} \prec \calM$
such that $A \subseteq N$ and $\norm{N}<\ka$.
\end{enumerate}
\end{prop}
\begin{proof}
(2) $\Rightarrow$ (1). Take $\gamma<\ka$, $\alpha\ge \ka$, and $x \in V_\alpha$.
Identifying $x$ as a constant,
consider the structure $\seq{V_\alpha;\in, x}$.
By (2), we can find an elementary submodel $X \prec V_\alpha$
such that $x \in X$, $V_\gamma \subseteq X$, and $\norm{X}<\ka$.
If $Y$ is the transitive collapse of $X$,
then $Y$ must be in $V_\ka$.

(1) $\Rightarrow$ (2).  $\norm{\ka}=\ka$ is already noted above.
Fix a structure $\calM=\seq{M;\ldots}$ and $A \subseteq M$
such that there is a surjection $f$ from some $x \in V_\ka$ onto $A$.
Take a large $\alpha>\ka$ with $f, \calM \in V_\alpha$,
and take also a large $\gamma<\ka$ with $x \in V_\gamma$.
By (1), we can find $X \prec V_\alpha$
such that $f, \calM \in V_\alpha$, $V_\gamma \subseteq X$,
and the transitive collapse of $X$ is in $V_\ka$. Since $x \in V_\gamma \subseteq X$,
we have $A=f``x \subseteq X$.
Let $\calN=\seq{M \cap X;\ldots}$.
Since $A \subseteq X \prec V_\alpha$,
it is easy to see that $\calN \prec \calM$ and $A \subseteq  M \cap X$.
If $\pi:X \to Y$ is the collapsing map,
then $\pi \restriction M \cap X: M \cap X \to \pi(M)$ is a bijection
and $\pi(M) \in V_\ka$.
Then the inverse map gives a surjection from $\pi(M)$ onto $M \cap X$.
Hence $\norm{M \cap X}<\ka$.
\end{proof}

\section{On  successors of weakly LS cardinals}

In $\ZFC$,
for every cardinal $\ka$ and $\alpha \ge \ka^+$,
we can find an elementary submodel $X \prec V_\alpha$
with $\size{X}=\ka \subseteq X$ and $X \cap \ka^+ \in \ka^+$.
We can prove a similar result for singular weakly LS cardinals.
\begin{lemma}\label{1}
Let $\ka$ be a singular weakly LS cardinal.
Then for every $\alpha \ge \ka^+$ and $x \in V_\alpha$,
there is $X \prec V_\alpha$ such that
$x \in X$, $V_\ka \subseteq X$, $X \cap \ka^+ \in \ka^+$,
and $\norm{X}=\ka$.
\end{lemma}
\begin{proof}
Fix a large $\beta>\alpha$,
and take $Y \prec V_\beta$ such that
$\alpha, \ka, x \in Y$, $\cf(\ka) \subseteq Y$,
and the transitive collapse of $Y$ belongs to $V_\ka$.

Let $\calF=\{Z \prec V_\alpha \mid x \in Z$ and the transitive collapse of $Z$ is in $V_\ka\}$.
We have $\calF \in Y$.
\begin{claim}
\begin{enumerate}
\item $\calF \cap Y$ is upward directed,
that is, for every $Z_0, Z_1 \in \calF\cap Y$,
there is $Z_2 \in \calF \cap Y$ with $Z_0 \cup Z_1 \subseteq Z_2$.
\item For every $\gamma<\ka$,
there is $Z \in \calF \cap Y$ with $V_\gamma \subseteq Z$.
\end{enumerate}
\end{claim}
\begin{proof}[Proof of Claim]
(1) Take a large $\gamma<\ka$ such that the transitive collapses of $Z_0$ and $Z_1$ are in $V_\gamma$,
and take  surjections $f_0:V_\gamma \to Z_0$ and $f_1:V_\gamma \to Z_1$.
Since $Z_0, Z_1 \in Y \prec V_\beta$, we may assume that
$\gamma, f_0, f_1 \in Y$.
In $Y$, we can choose $Z_2' \prec V_{\alpha+\om}$
such that $V_\gamma \subseteq Z_2'$, $x, \alpha, f_0,f_1 \in Z_2'$,
and the transitive collapse of $Z_2'$ is in $V_\ka$.
Let $Z_2=Z_2' \cap V_\alpha \in Y$. We know $Z_2 \prec V_\alpha$, so
$Z_2 \in \calF \cap Y$.
Since $V_\gamma \subseteq Z_2$, we also have $Z_0 \cup Z_1 \subseteq Z_2$.

(2) Since $\cf(\ka) \subseteq Y$, we have that $Y \cap \ka$ is cofinal in $\ka$.
Hence for a given $\gamma<\ka$, there is $\delta \in Y \cap \ka$ with $\gamma \le \delta$.
Then in $Y$ we can choose $Z \in \calF$ with $V_\delta \subseteq Z$.
\end{proof}

Let $X=\bigcup(\calF \cap Y)$. By the claim above,
we have that $x \in X \prec V_\alpha$ and $V_\ka \subseteq X$.

\begin{claim}
There is a surjection from $V_\ka$ onto $X$.
\end{claim}
\begin{proof}[Proof of Claim]
Let $\overline{Y}$ be the transitive collapse of $Y$,
and $\pi:Y \to \overline{Y}$ the collapsing map.
For each $Z \in \calF \cap Y$,
let $\overline{Z}$ be the transitive collapse of $Z$, and
$\pi_z$ be the collapsing map.

Define $f:V_\ka \times V_\ka \to X$ as follows:
For $\seq{a, b} \in V_\ka \times V_\ka$,
if $\pi^{-1}(a) \in \calF \cap Y$ and
$b \in \overline{Z}$ (where $Z=\pi^{-1}(a)$),
then $f(a,b)=\pi^{-1}_{Z}(b)$.
Otherwise, $f(a,b)=\emptyset$.
This $f$ is an surjection; Take $c \in X$.
Then $c \in Z$ for some $Z \in \calF \cap Y$,
and $f(\pi(Z), \pi_Z(c))=c$.
We can easily take a surjection from $V_\ka$ onto $V_\ka \times V_\ka$,
hence we obtain a surjection from $V_\ka$ onto $X$.
\end{proof}

Finally, since there is no cofinal map from $V_\ka$ into $\ka^+$ by Theorem \ref{1.2},
we have that $\sup(X \cap \ka^+)<\ka^+$.
In addition, since $\ka \subseteq X$, we have
$\sup(X \cap \ka^+) \subseteq X$ and $\sup(X \cap \ka^+)=X \cap \ka^+ \in \ka^+$.
\end{proof}

Let $\ka$ be a singular weakly LS cardinal.
While we already knew that 
the club filter over $\ka^+$ is $\ka^+$-complete,
we do not know if it is normal.
Among this, we can construct a normal filter over $\ka^+$ which is 
definable with parameter $\ka^+$.

\begin{note}
Let $\ka$ be a cardinal,
and $F$ a filter over $\ka$.
Then the following are equivalent:
\begin{enumerate}
\item For every $X_\alpha \in F$ $(\alpha<\ka)$,
the diagonal intersection 
$\triangle_{\alpha<\ka} X_\alpha
=\{\beta<\ka \mid \beta \in X_\alpha$ for all $\alpha<\beta\}$ is in $F$.
\item For every $X \in F^+$ and  regressive function $f:X \to \ka$,
there is $\alpha<\ka$ with $\{\beta \in X \mid f(\beta)=\alpha\} \in F^+$.
\end{enumerate}
Where $F^+=\{X \in \p(\ka) \mid X \cap C \neq \emptyset$ for every $C \in F\}$.
An element of $F^+$ is an \emph{$F$-positive set}.
We say that a filter $F$ is \emph{normal}
if $F$ is proper, contains all co-bounded subsets of $\ka$, 
and satisfies the above conditions (1) and/or (2).
\end{note}

\begin{prop}
Let $\ka$ be a singular weakly LS cardinal.
Let $F \subseteq \p(\ka^+)$ be the set such that:
$D \in F \iff$ there is $\alpha \ge \ka^+$ and $x \in V_\alpha$ such that
$D$ contains the set $\{\eta<\ka^+ \mid$ there is $X \prec V_\alpha$ with
$x \in X$ and $\eta =X \cap \ka^+\}$.
Then $F$ is a normal filter over $\ka^+$.
\end{prop}
\begin{proof}
By Lemma \ref{1}, we have that 
$\emptyset \notin F$ and $\ka^+ \setminus \eta \in F$ for every $\eta<\ka^+$.
One can check that $F$ is a filter over $\ka^+$.
For the normality,
take a family $\{D_\gamma \in F \mid \gamma<\ka^+\}$,
and let $D=\triangle_{\gamma<\ka^+} D_\gamma$.
Suppose to the contrary that $D \notin F$.
Now fix a large $\beta>\ka^+$ such that
for every $\gamma<\ka^+$, there is $\alpha <\beta$ and $x \in V_\alpha$ such that
$\{\eta<\ka^+ \mid$ there is $X \prec V_\alpha$ with
$x \in X$ and $\eta =X \cap \ka^+\} \subseteq D_\gamma$.
Since $D \notin F$,
we can find $Y \prec V_\beta$ such that $Y$ contains all relevant objects,
$Y \cap \ka^+ \in \ka^+$, and $Y \cap \ka^+ \notin D$. Let $\delta=Y \cap \ka^+$.
We see that $\delta \in D_\gamma$ for every $\gamma<\delta$, this is a contradiction.

Take $\gamma<\delta$. Then $\gamma \in Y$.
Hence we can find $\alpha \in Y $ and $x \in Y\cap V_\alpha$
such that
$\{\eta<\ka^+ \mid$ there is $X \prec V_\alpha$ with
$x \in X$ and $\eta =X \cap \ka^+\} \subseteq D_\gamma$.
However, since $\alpha, x \in Y$, we have $x \in Y \cap V_\alpha \prec V_\alpha$
and $(Y \cap V_\alpha) \cap \ka^+=\delta$,
hence $\delta \in D_\gamma$.
\end{proof}

For a non-empty set $S$,
let $\col(S)$ be the poset of all finite partial functions from $\om$ to $S$ with the 
reverse inclusion order.
The forcing with $\col(S)$ adds a surjection from $\om$ onto $S$.

\begin{prop}
Let $\ka$ be a singular weakly LS cardinal.
Then $\col(V_\ka)$ forces $(\ka^+)^V=\om_1$ and the Dependent Choice $\DC$.
\end{prop}
\begin{proof}
For the equality $(\ka^+)^V=\om_1$,
take $p \in \col(V_\ka)$ and a name $\dot f$ for a function from $\ka$ to $\ka^+$.
Define $F:\col(V_\ka) \times \ka \to \ka^+$ as follows:
For $\seq{q, \alpha} \in \col(V_\ka) \times \ka$,
if $q \le p$ and $q \Vdash \dot f(\alpha)=\eta$'' for $\eta<\ka^+$,
then $F(q, \alpha)=\eta$.
Otherwise, let $F(q, \alpha)=0$.
By Theorem \ref{1.2}, $F$ is not a cofinal map,
hence $\gamma=\sup(F``(\col(V_\ka) \times \ka))<\ka^+$.
Then it is clear that $p \Vdash$``$\dot f``\ka \subseteq \gamma$'',
so $p \Vdash$``$(\ka^+)^V$ is regular''.

To show that $\col(V_\ka)$ forces $\DC$,
take $p \in \col(V_\ka)$, $\col(V_\ka)$-names $\dot S$ and $\dot R$
such that $p \Vdash$``$\dot S$ is non-empty, $\dot R \subseteq \dot S^2$,
and for every $x \in \dot S$ there is $y \in \dot S$ with $\seq{x,y} \in \dot R$''.

Fix a large limit $\alpha>\ka^+$ with $\dot S, \dot R \in V_\alpha$.
By Lemma \ref{1},
we can find $X \prec V_\alpha$ such that
$V_\ka \subseteq X$, $X$ contains all relevant objects, $X \cap \ka^+ \in \ka^+$,
and $\norm{X}=\ka$.
Note that $\col(V_\ka) \subseteq X$.

Take a $(V, \col(V_\ka))$-generic $G$,
and let $X[G]=\{\dot x_G \mid \dot x \in X$ is a $\col(V_\ka)$-name$\}$
(where $\dot x_G$ is the interpretation of $\dot x$ by $G$).
Since $\col(V_\ka) \subseteq X$, we may assume $X[G] \prec V_\alpha[G]=V[G]_\alpha$.
There is a canonical surjection from $X$ onto $X[G]$,
namely $\dot x \mapsto \dot x_G$, 
hence we can take a surjection from $V_\ka$ onto $X[G]$.
Since $V_\ka$ is countable in $V[G]$,
we have that $X[G]$ is countable as well.
Let $S=\dot S_G$ and $R=\dot R_G$. We know $S, R \in X[G]$.
By the elementarity of $X[G]$, for every $x \in S \cap X[G]$,
there is $y \in S \cap X[G]$ with $\seq{x,y} \in R$.
Now $X[G]$ is well-orderable,
thus we can take a map $f:\om \to S \cap X[G]$ 
such that $\seq{f(n), f(n+1)} \in R$ for every $n<\om$.
\end{proof}

When $\ka$ is regular weakly LS,
we can obtain a similar result to Lemma \ref{1}.
Note that  in $\ZFC$, $\ka$ is regular weakly LS if and only if $\ka$ is inaccessible.
\begin{lemma}\label{reg}
Let $\ka$ be a regular weakly LS cardinal.
Then for every $\gamma<\ka$, $\alpha >\ka$, and $x \in V_\alpha$,
there is $X \prec V_\alpha$ such that 
$x \in X$, $\gamma<X \cap \ka \in \ka$, $V_{X \cap \ka} \subseteq X$,
and the transitive collapse of $X$ is in $V_\ka$.
\end{lemma}
\begin{proof}
Similar to Lemma \ref{1}.
Take a large $\beta>\alpha$
and $Y \prec X_\beta$ such that $Y$ contains all relevant objects, and
the transitive collapse of $Y$ is in $V_\ka$.
We know $\sup(Y \cap \ka)<\ka$;
If $\sup(Y \cap \ka)=\ka$, then $\ot(Y \cap \ka)=\ka$ since $\ka$ is regular.
However then the transitive collapse of $Y$ cannot be in $V_\ka$.

Let $\calF=\{Z \prec V_\alpha\mid x \in Z$, the transitive collapse of $Z$ is in $V_\ka\}$,
and $X=\bigcup(\calF \cap Y)$. We have $X \prec V_\alpha$.
For every $Z \in \calF \cap Y$, we have that $Z \cap \ka\le \sup(Y \cap \ka)$,
and for every $\delta \in Y \cap \ka$,
there is $Z \in \calF \cap Y$ with $V_\delta \subseteq Z$.
Hence $X \cap \ka =\sup(Y \cap \ka)$ and $V_{X \cap \ka} \subseteq X$.

Take a large $\delta<\ka$ such that the transitive collapse of $Y$ is in $V_\delta$.
Note that $\sup(Y \cap \ka) \le \delta$.
For every $Z \in \calF \cap Y$, the transitive collapse of $Z$ is also in $V_{\delta}$.
Then as in Lemma \ref{1} we can define a surjection from $V_\delta$ onto $X$.
Therefore $\norm{X} \le \delta<\ka$, and the transitive collapse of $X$ is in $V_\ka$.
\end{proof}

\section{On  successors of LS cardinals}
If $\ka$ is a singular LS cardinal and there is a weakly LS cardinal $ \le \cf(\ka)$,
we can take a small elementary submodel with certain closure property.
We use the following fact:
\begin{thm}[\cite{U}]\label{fact3.7}
Let $\ka$ be a weakly LS cardinal.
Then for every cardinal $\la \ge \ka$ and $x \in V_\ka$,
there is no cofinal map from $x$ into $\la^+$.
\end{thm}

\begin{lemma}\label{3.22}
Let $\alpha$ be a limit ordinal $\alpha$ and a set $x$,
if $\cf(\alpha)>\norm{x}$ and there is a weakly LS cardinal $\ka$
with $\cf(\alpha) \ge \ka \ge \norm{x}$,
then 
there is no cofinal map from $x$ into $\alpha$.
\end{lemma}
\begin{proof}
First suppose $\norm{x}<\ka<\cf(\alpha)$.
In this case we may assume $x \in V_\ka$.
If $\cf(\alpha)$ is a successor cardinal,
then we have done by Theorem \ref{fact3.7}.
Suppose $\cf(\alpha)$ is a limit cardinal.
Then $\cf(\alpha)>\ka^+$, so $\{\eta<\alpha \mid \cf(\eta)=\cf(\ka^+)\}$ is stationary in $\alpha$.
If $f:x \to \alpha$ is a cofinal map,
then $C=\{\eta<\alpha \mid \eta \cap f``x$ is cofinal in $\eta\}$ is a club in $\alpha$.
Hence we can find $\eta<\alpha$ such that $\eta \cap f``x$ is cofinal in $\eta$
and $\cf(\eta)=\cf(\ka^+)$.
However then there is a cofinal map from $x$ into $\ka^+$,
This is a contradiction.

Next suppose $\norm{x}=\ka<\cf(\alpha)$.
Then we may assume $x=V_\ka$.
Take $f:V_\ka \to \alpha$.
By the previous case, for every $\gamma<\ka$,
we have $\sup(f``V_\gamma)<\alpha$.
Since $\cf(\alpha)>\ka$, we have $\sup(f``V_\ka)=\sup_{\gamma<\ka} \sup(f``V_\gamma)<\alpha$.

Finally suppose $\norm{x}<\ka =\cf(\alpha)$.
Then $\ka$ is a regular weakly LS cardinal.
If there is a cofinal map from $x$ into $\alpha$,
 we can find a cofinal map $f$ from $x$ into $\ka$.
Take an elementary submodel $X \prec V_{\alpha+\om}$
such that $V_{\norm{x}} \subseteq X$, $X$ contains all relevant objects,
and the transitive collapse of $X$ is in $V_\ka$.
We know $f``x \subseteq X$, hence $\sup(X \cap \ka)=\ka$.
Then $\ot(X \cap \ka)=\ka$ because $\ka$ is regular,
so the transitive collapse of $X$ cannot be in $V_\ka$.
This is a contradiction.
\end{proof}

\begin{lemma}\label{3.8}
Let $\ka$ be a singular LS cardinal and $\nu \le \cf(\ka)$  a weakly LS cardinal.
Then for every $\alpha>\ka$ and $x \in V_\alpha$,
there is $\beta>\alpha$ and $X \prec V_\beta$
such that $V_\ka \subseteq X$, $x \in X$, 
$X \cap \ka^+ \in \ka^+$, ${}^{V_\gamma} (X \cap V_\alpha)\subseteq X$ for every $\gamma<\nu$,
and $\norm{X}=\ka$.
\end{lemma}
\begin{proof}
Take a large $\beta$ and $Y \prec V_{\beta}$
such that $\alpha, x,\dotsc \in Y$, ${}^{V_\nu} (Y \cap V_{\alpha+\om+\om}) \subseteq Y$,
$\cf(\ka) \subseteq Y$, and the transitive collapse of $Y$ is in $V_\ka$.
Let $\calF=\{Z \prec V_{\alpha+\om} \mid x \in Z$, the transitive collapse of $Z$ is in $V_\ka\}$.
Let $X=\bigcup(\calF \cap Y)$. 
As in the proof of Lemma \ref{1}, we have $X \prec V_{\alpha+\om}$,
$x \in X$, $V_\ka \subseteq X$, $X \cap \ka^+ \in \ka^+$, and $\norm{X}=\ka$.

We have to see that ${}^{V_\gamma} (X \cap V_\alpha) \subseteq X$ for every $\gamma<\nu$.
Take $f:V_\gamma \to X \cap V_\alpha$.
Since $\nu$ is weakly LS, we can find 
$Y' \prec V_{\beta}$ such that $Y$ contains all relevant objects,
${V_\gamma} \subseteq Y'$, and the transitive collapse of $Y'$ is in $V_\nu$.
We have $f \subseteq Y'$.
Let $\calF'=\calF \cap Y \cap Y' \subseteq Y \cap V_{\alpha+\om+\om}$. 
Since ${}^{V_\nu} (Y \cap V_{\alpha+\om+\om}) \subseteq Y$ and $\norm{\calF'} \le \norm{Y}<\nu$,
we have $\calF' \in Y$.
Note that $\norm{Z}<\ka$ for every $Z \in \calF'$.
Since $\norm{\calF'}<\nu \le \cf(\ka)$,
there is no cofinal map from $\calF'$ into $\ka$ by Lemma \ref{3.22}.
Thus we have $\delta=\sup\{\norm{Z} \mid Z \in \calF'\}<\ka$.
Then, as in the proof of Lemma \ref{1}, 
we can define a surjection from $V_\delta$ onto $\bigcup \calF'$.
In $Y$, we can find $Z' \in \calF \cap Y$ such that
$\p(V_\alpha \cap \bigcup \calF') \subseteq Z'$.
Since $f \in Y'$, for each $a \in V_\gamma$
there is $Z \in (\calF \cap Y) \cap Y'$ with $f(a) \in Z$.
Hence $f``V_\gamma \subseteq V_\alpha \cap \bigcup \calF'$, and
$f``V_\gamma \in Z' \subseteq X$.
\end{proof}
\begin{note}
In Lemma \ref{3.8},
if $\nu<\cf(\ka)$ is a singular weakly LS cardinal,
then we can require that ${}^{V_\nu} (X \cap V_\alpha) \subseteq X$.
\end{note}

If $\ka$ is regular LS cardinal, we have the following parallel result:
\begin{lemma}\label{3.8}
Let $\ka$ be an LS cardinal and $\nu <\ka$ a weakly LS cardinal.
Then for every $\alpha>\ka$, $x \in V_\alpha$, and $\gamma<\ka$,
there is $\beta \ge \alpha$ and $X \prec V_\beta$
such that $x \in X$, $\gamma<X \cap \ka \in \ka$,
$V_{X \cap \ka} \subseteq X$, ${}^{V_\delta} (X \cap V_\alpha)\subseteq X$ for every $\delta<\nu$,
and the transitive collapse of $X$ is in $V_\ka$.
\end{lemma}

For an ordinal $\gamma$,
let $\mathsf{DC}_{\gamma}$ be the assertion 
that for every non-empty set $S$ and
$G:{}^{<\gamma} S \to \p(S) \setminus \{\emptyset\}$,
there is $f:\gamma \to S$ such that
$f(\alpha) \in G(f \restriction \alpha)$ for every $\alpha<\gamma$.
Let $\mathsf{DC}_{<\gamma}$ be the assertion that
$\mathsf{DC}_\delta$ holds for every $\delta<\gamma$.

$\mathsf{DC}_\om$ is the Dependent Choice, and 
it is known that $\AC$ is equivalent 
to that $\mathsf{DC}_\gamma$ for every $\gamma$.
Note also that the following:
\begin{enumerate}
\item $\mathsf{DC}_\gamma \Rightarrow \mathsf{DC}_{\gamma+1}$.
\item If $\gamma$ is a singular ordinal and $\mathsf{DC}_{<\gamma}$ holds,
then $\mathsf{DC}_\gamma$ holds as well.
\item $\mathsf{DC}_\gamma \Rightarrow \mathsf{DC}_{<\ga^+}$.
\end{enumerate}

Woodin \cite{W} proved that collapsing a supercompact cardinal
yields $\mathsf{DC}_{\ka}$ for some large $\ka$.
Where $\ka$ is \emph{supercompact}
if for every $\alpha > \ka$,
there is $\beta \ge \alpha$, a transitive set $N$ with ${}^{V_\alpha} N \subseteq N$,
and an elementary embedding $j:V_\beta \to N$
with critical point $\ka$ and $\alpha<j(\ka)$.
It is known that a supercompact cardinal is a limit of LS cardinals
(\cite{U}).

We prove a similar result using LS cardinals.

\begin{define}
Let $\ka$ be a regular cardinal
and $S$ a non-empty set.
Let $\col(\ka, S)$ be the 
poset of all partial functions $p$ from $\ka$ to $S$ with
$\size{p}<\ka$ (so $p$ is assumed to be well-orderable).
The ordering of $\col(\ka, S)$ is the reverse inclusion.

If $G$ is $(V, \col(\ka, S))$-generic,
then $\bigcup G$ is a surjection from $\ka$ onto $S$.
\end{define}
\begin{prop}
Let $\ka$ be a regular uncountable cardinal,
and suppose $\mathsf{DC}_{<\ka}$ holds.
Let $\la>\ka$ be a singular LS cardinal such that $\cf(\la)>\ka$ and
there is a weakly LS cardinal $\nu$ with $\ka \le \nu \le \cf(\la)$.
Then $\col(\ka, V_\la)$ does not add new $<\ka$-sequences, and 
forces $(\la^+)^V=(\ka^+)^{V^{\col(\ka, V_\gamma)}}$ and $\mathsf{DC}_{\ka}$.
\end{prop}
\begin{proof}
By $\mathsf{DC}_{<\ka}$,
we know that $\col(\ka, V_\la)$ is $\ka$-closed
and does not add new $<\ka$-sequences.
We can check that $\col(\ka, V_\la)$ preserves $\mathsf{DC}_{<\ka}$ as well.

To see $\col(\ka, V_\la)$ forces $(\la^+)^V=(\ka^+)^{V^{\col(\ka, V_\la)}}$,
it is clear that $(\la^+)^V \le (\ka^+)^{V^{\col(\ka, V_\la)}}$.
For the converse, take $p \in \col(\ka, V_\la)$ and a name $\dot f$ for
a function from $\la$ to $(\la^+)^V$.
Define $F:\la \times \col(\ka, V_\la) \to \la^+$
by $F(\delta, q)=\eta \iff q \le p$ and $q \Vdash$``$\dot f(\delta)=\eta$''.
$F$ is not cofinal by Theorem \ref{1.2},
hence the image of $F$ has an upper bound in $\la^+$, say $\eta$.
It is clear that $p \Vdash$``$\dot f`` \la\subseteq \eta$'',
so $\Vdash$``$(\la^+)^V$ is regular''.

For $\mathsf{DC}_{\ka}$,
take $p \in \col(\ka, V_\la)$ and
names $\dot S$ and $\dot F$ with $p \Vdash$``$\dot F:{}^{<\ka} \dot S \to \p(\dot S)\setminus \{\emptyset\}$''.
Fix a sufficiently large $\alpha>\la^+$.
By Lemma \ref{3.8}, we can find $\beta>\alpha$ and $X \prec V_\beta$
such that $X$ contains all relevant objects, $V_\la \subseteq X$,
$X \cap \la^+ \in \la^+$, ${}^{V_\gamma}(X \cap V_{\alpha+\om}) \subseteq X$ for every $\gamma<\nu$,
and $\norm{X}=\la$.
Take a $(V, \col(\ka, V_\la))$-generic $G$ and work in $V[G]$.
We may assume $X[G] \prec V[G]_\beta$.
We see that $X[G] \cap V[G]_\alpha$ is closed under $<\ka$-sequences in $V[G]$.
Take $\eta<\ka$ and $f:\eta \to X[G] \cap V[G]_\alpha$.
By $\mathsf{DC}_{<\ka}$ in $V[G]$,
we can find a sequence of $\col(\ka,V_\la)$-names $\seq{\dot x_i \mid i<\eta}$
such that $\dot x_i \in X$ and $f(i)=(\dot x_i)_G$ for every $i<\eta$.
We may assume $\dot x_i \in V_{\alpha+\om}$ for $i<\eta$.
Since $\col(\ka, V_\la)$ does not add new $<\ka$-sequences,
we have $\seq{\dot x_i\mid i<\eta} \in V$,
hence $\seq{ \dot x_i \mid i<\eta} \in X$ and
$\{(\dot x_i)_G \mid i<\eta\} \in X[G]$.

Since there is a sujrction from $\ka$ onto $V_\la$, from $V_\la$ onto $X$, and
from $X$ onto $X[G]$, we can obtain a bijection from $\ka$ onto $X[G]$.
Let $S=\dot S_G$ and $F=\dot F_G$.
Since $X[G] \prec V[G]_\beta$, $X[G] \cap V[G]_\alpha$ is closed under $<\ka$-sequences, and
$X[G]$ is well-orderable, 
we can take $f:\ka \to S \cap X[G]$ such that
$f(i) \in F(f \restriction i)$ for every $i<\ka$.
\end{proof}

If there are proper class many LS cardinals,
then for every cardinal $\ka$,
there is a singular LS cardinal $\la>\ka$ 
such that there is an LS cardinal $\nu$ with $\ka<\nu<\cf(\la)$;
Take a singular LS cardinal $\nu>\ka$, and let $\la$ be the $\nu^+$-th LS cardinal above $\nu$.
The cofinality of $\la$ is $\nu^+$.
In addition, in \cite{U} we proved that if $\ka$ is a limit of LS cardinals,
then every poset with rank $<\ka$ forces that $\ka$ is LS.
Using these facts,
we can obtain the following corollary:
\begin{cor}
Suppose there are proper class many LS cardinals.
Then there is a definable class forcing which forces
$\ZFC$.
\end{cor}

\section{LS cardinals in $\HOD(V_\la)$}

For a set $X$,
let $\HOD(X)$ be the class of all  hereditarily definable sets 
with parameters from $\mathrm{ON} \cup \mathrm{trcl}(\{X\})$.
$\HOD(X)$ is a transitive model of $\ZF$ with $X \in \HOD(X)$.
By the definition of $\HOD(X)$,
for every $x \in \HOD(X)$
there is an ordinal $\theta$ and a surjection $\sigma:\theta \times {}^{<\om} (\mathrm{trcl}(\{X\}) \to x$
with $\sigma \in \HOD(X)$.
If $X=V_\alpha$ and $\alpha$ is limit, then $V_\alpha$ is transitive, ordinal definable, and 
${}^{<\om} V_\alpha \subseteq V_\alpha$.
Hence the domain of $\sigma$ can be the set $\theta \times  V_\alpha$ in this case.

\begin{lemma}\label{5.1}
Let $\delta$ be a limit ordinal.
Then for every limit ordinals $\alpha>\beta \ge \delta$ and  $x \in V_\alpha \cap \HOD(V_\delta)$,
there is $X \prec V_\alpha \cap \HOD(V_\delta)$
such that $x \in X \in \HOD(V_\delta)$,
$V_\beta \cap \HOD(V_\delta)\subseteq X$,
and $\norm{X}^{\HOD(V_\delta)}=\beta$.
\end{lemma}
\begin{proof}
For given $\alpha>\beta \ge \delta$,
we can take a large ordinal $\theta$ and a surjection $\sigma:\theta \times V_\delta
\to V_\alpha \cap \HOD(V_\delta)$ with $\sigma \in \HOD(V_\delta)$.
By induction on $n<\om$, we define  $X_n$ and $f_n$ as follows.
First, let $X_0=(V_\beta \cap \HOD(V_\delta)) \cup \{x\}$, and take a surjection $f_0:V_\beta
\cap \HOD(V_\delta)
 \to X_0$ with
$f_0 \in \HOD(V_\delta)$.
Suppose $X_n$ and $f_n$ are defined and $f_n$ is a surjection from $V_\beta\cap \HOD(V_\delta)$ onto $X_n$.
For a formula $\varphi(v_1,\dotsc, v_k,w)$ and $x_1,\dotsc, x_k \in X_n$,
if $\exists w \varphi(x_1,\dotsc, x_k,w)$ holds in $V_\alpha \cap \HOD(V_\delta)$
then we can find the least $\alpha_{\varphi, x_1,\dotsc, x_k} \in \theta$
such that $\exists w \in \sigma``(\{\alpha_{\varphi, x_1,\dotsc, x_k}\} \times V_\delta)\, \varphi(x_1,\dotsc, x_k,w)$ holds in $V_\alpha \cap \HOD(V_\delta)$.
Set $X_{n+1}=X_n \cup \bigcup\{ \sigma``(\{\alpha_{\varphi,x_1,\dotsc, x_k}\} \times V_\delta):
\varphi$ is a formula, $x_1,\dotsc, x_k \in X_n\}$.
There is a canonical surjection from $V_\beta\cap \HOD(V_\delta)$ onto $
\bigcup\{ \sigma``(\{\alpha_{\varphi,x_1,\dotsc, x_k}\} \times V_\delta):
\varphi$ is a formula, $x_1,\dotsc, x_k \in X_n\}$, namely
$\seq{\lceil \varphi \rceil, y_1,\dotsc, y_k,z} \mapsto 
\sigma(\alpha_{\varphi,f_n(y_1),\dotsc, f_n(y_k)},z)$
(where $\lceil \varphi \rceil \in \om$ is the G\"odel number of $\varphi$).
Hence we can define $f_{n+1} :V_\beta \cap \HOD(V_\delta)\to X_{n+1} $ canonically using this surjection and $f_n$.
We can carry out this construction in $\HOD(V_\delta)$, hence $\seq{X_n,f_n \mid n<\om}
\in \HOD(V_\delta)$.
Thus we have $X=\bigcup_n X_n \in \HOD(V_\delta)$. We know $V_\beta \cap \HOD(V_\delta)\subseteq X$ and $x \in X$.
By Tarski-Vaught criterion, we have $X \prec V_\alpha \cap \HOD(V_\delta)$.
Moreover, there is a surjection from $\om \times (V_\beta\cap \HOD(V_\delta))$ onto $X$ in $\HOD(V_\delta)$
constructed from $\seq{f_n \mid n<\om}$,
so we can construct a surjection $f:V_\beta \cap \HOD(V_\delta)\to X$ in $\HOD(V_\delta)$,
hence $\norm{X}^{\HOD(V_\delta)}=\beta$.
\end{proof}

The following is immediate from the previous lemma:
\begin{cor}
Let $\delta$ be a limit ordinal.
Then for every cardinal $\ka>\delta$ with $\norm{\ka}=\ka$,
$\ka$ is weakly LS in $\HOD(V_\delta)$.
\end{cor}
\begin{proof}
Take $\gamma<\ka$, $\alpha>\ka$ and $x \in V_\alpha \cap \HOD(V_\delta)$.
We may assume $\alpha$ and $\gamma$ are limit ordinals and $\gamma \ge \delta$.
By the previous lemma,
there is $X \prec V_\alpha \cap \HOD(V_\delta)$ such that
$x \in X \in \HOD(V_\delta)$, $V_\gamma \cap \HOD(V_\delta) \subseteq X$, and $\norm{X}^{\HOD(V_\delta)}=\gamma$.
Then $\norm{X} \le \norm{X}^{\HOD(V_\delta)}=\gamma<\ka$,
so the transitive collapse of $X$ must be in $V_\ka$.
\end{proof}

An uncountable cardinal $\ka$ is said to be \emph{inaccessible}
if for every $\gamma<\ka$, there is no cofinal map from $V_\gamma$ into $\ka$.
Every inaccessible cardinal is regular, and 
in $\ZFC$, this definition is equivalent to the standard one.
One can check that every regular weakly LS cardinal is inaccessible.

The following corollary is already proved in 
Schlutzenberg \cite{S}.
\begin{cor}
Let $\ka$ be an uncountable cardinal,
and suppose $\ka$ is inaccessible in $\HOD(V_\ka)$.
Then $\ka$ is a weakly LS cardinal in $\HOD(V_\ka)$.
In particular, if $\ka$ is inaccessible then
$\ka$ is regular weakly LS  in $\HOD(V_\ka)$.
\end{cor}
\begin{proof}
Take $\gamma<\ka$, $\alpha>\ka$ and $x \in V_\alpha \cap \HOD(V_\ka)$.
By Lemma \ref{5.1},
we can find $X \prec V_\alpha \cap \HOD(V_\ka)$
such that $x \in X \in \HOD(V_\ka)$, 
$V_\ka \subseteq X$, 
and $\norm{X}^{\HOD(V_\ka)}=\ka$.
Take a surjection $\sigma:V_\ka \to X$ in $\HOD(V_\ka)$.
Since $\ka$ is inaccessible in $\HOD(V_\ka)$,
the set $\{\eta<\ka \mid \sigma``V_\eta \prec V_\alpha \cap \HOD(V_\ka)\}$
contains a club in $\ka$,
hence we can find $\eta<\ka$ such that
$V_\gamma \subseteq \sigma``V_\eta \prec V_\alpha \cap \HOD(V_\ka)$.
\end{proof}
\begin{note}
For an uncountable cardinal $\ka$,
if $\ka$ is not regular in $\HOD(V_\ka)$
then $\cf(\ka)^{\HOD(V_\ka)}=\cf(\ka)$.
\end{note}

\begin{cor}
Let $\delta$ be a limit ordinal
and $\nu \le \cf(\delta)$ a weakly LS cardinal.
Then $\nu$ is weakly LS in $\HOD(V_\delta)$.
\end{cor}
\begin{proof}
For $\gamma<\nu$, $\alpha>\delta$ and $x \in V_\nu \cap \HOD(V_\delta)$,
we can take $X \prec V_\alpha \cap \HOD(V_\delta)$
such that $x \in X \in \HOD(V_\delta)$, 
$V_\delta \subseteq X$, 
and there is a surjection $\sigma:V_\delta \to X$ in $\HOD(V_\delta)$.
Define $R_0, R_1 \subseteq V_\delta^2$ by
$a \mathrel{R_0} b \iff \sigma(a) \in \sigma(b)$, and
$a \mathrel{R_1} b \iff \sigma(a) = \sigma(b)$.
We know $R_0, R_1 \in \HOD(V_\delta)$.
Since $\nu$ is weakly LS,
we can find $\seq{Y; \in, R_0 \cap Y,R_1 \cap Y} \prec
\seq{V_\delta; \in, R_0, R_1}$
such that $V_\gamma \subseteq Y$, $\norm{Y}<\nu$,
and there is $a \in Y$ with $\sigma(a)=x$.
Since $\norm{Y}<\nu \le \cf(\delta)$,
there is no cofinal map from $Y$ to $\delta$.
In particular $\sup\{\mathrm{rank}(y) \mid y \in Y\}<\delta$, 
so $Y \in V_\delta$ and $Y \in \HOD(V_\delta)$.
Then one can check that $V_\gamma \subseteq \sigma``Y \prec V_\alpha \cap \HOD(V_\delta)$,
as required.
\end{proof}

When $\norm{\ka}=\ka$ and $\ka$ is singular in $\HOD(V_\ka)$, 
we do not know whether $\ka$ is weakly LS in $\HOD(V_\ka)$.

\begin{note}
$\HOD(V_\delta)$ and $\HOD(V_\ka)$ in this section can be replaced by
$L(V_\delta)$ and $L(V_\ka)$.
\end{note}

\section{On elemtary embeddings}
Woodin \cite{W} proved that if
$\ka$ is a singular limit of supercompact cardinals and
 there is a set $A$ of ordinals with $\ka^+=(\ka^+)^{L[A]}$,
then there is no non-trivial elementary embedding $j:V_{\ka+2} \to V_{\ka+2}$.
We can prove the same result using weakly LS cardinals.
Note that whenever $A$ is a set of ordinals,
$\HOD(A)$  is a transitive model of $\ZFC$ and $A \in \HOD(A)$.
 
\begin{prop}\label{3.1}
Suppose $\ka$ is a singular weakly LS cardinal.
If there is a set $A$ of ordinals with $\ka^+=(\ka^+)^{\HOD(A)}$,
then there is no non-trivial elementary embedding $j:V_{\ka+2} \to V_{\ka+2}$.
\end{prop}
\begin{proof}
Suppose such an elementary  $j:V_{\ka+2} \to V_{\ka+2}$ exists.
Let $F$ be the club filter over $\ka^+$ restricted to
the set $\{\alpha<\ka^+ \mid \cf(\alpha)=\om\}$,
it  is $\ka^+$-complete.
$F$ is ordinal definable, hence $F \cap \HOD(A) \in \HOD(A)$,
and $F \cap \HOD(A)$ is a $\ka^+$-complete filter in $\HOD(A)$.
Since $\HOD(A)$ computes $\ka^+$ correctly and is a model of $\ZFC$,
we have that $F \cap \HOD(A)$ is not $\ka^+$-saturated in $\HOD(A)$,
hence is not $\mathrm{crit}(j)$-saturated.
Take pairwise disjoint $F$-positive sets $\{E_\alpha \mid \alpha<\mathrm{crit}(j)\} \in \HOD(A)$.
Each $E_\alpha$ is a stationary set in $\ka^+$,
and we may assume that $E_\alpha
\subseteq \{\eta<\ka^+ \mid \cf(\eta)=\om\}$.
Hence we obtain $\mathrm{crit}(j)$-many pairwise disjoint stationary subsets
of $\{\eta<\ka^+ \mid \cf(\eta)=\om\}$.
This partition can be coded in $V_{\ka+2}$,
and using this coded partition we can derive the contradiction
(e.g., see Kanamori \cite{K}).
\end{proof}

The assumption in the previous proposition can be weakened as follows.
\begin{prop}
Suppose $\ka$ is a singular weakly LS cardinal.
If there is a set $A$ of ordinals such that $\ka^+$ is not measurable in $\HOD(A)$,
then there is no non-trivial elementary embedding $j:V_{\ka+2} \to V_{\ka+2}$.
\end{prop}
\begin{proof}
First note that $\ka$ is strong limit in $\HOD(A)$.
Let $F$ be the filter defined as in Proposition \ref{3.1}.
By Tarski's theorem,
if $F \cap \HOD(A)$ is $\mathrm{crit}(j)$-saturated in $\HOD(A)$,
then $\ka^+$ is measurable in $\HOD(A)$,
this contradicts to the assumption.
Hence $F$ cannot be $\mathrm{crit}(j)$-saturated in $\HOD(A)$,
and we can take a large stationary partition of
$\{\eta<\ka^+ \mid \cf(\eta)=\om\}$.
The rest is the same to before.
\end{proof}

\begin{cor}
Suppose there are proper class many weakly LS cardinals.
If there is a non-trivial elementary embedding $j:V \to V$,
then $\{\ka \mid$ $\ka$ is singular and $\ka^+$ is measurable in $\HOD\}$
forms a proper class.
\end{cor}

\section{On ultrapowers}
A cardinal $\ka$ is said to be \emph{critical}
if there is a transitive class $M$ and an elementary embedding $j:V \to M$ with
critical point $\ka$ (Schlutzenberg \cite{S}).

\begin{note}
The term \emph{critical} was used in the first draft of \cite{S}, 
but this term was already used in Hayut-Karagila \cite{HK} with another definition.
The latest version of \cite{S} uses the term \emph{$V$-critical} instead of
\emph{critical} (see also \cite{S2}).
\end{note}

While the definition of critical cardinal is a second-order statement,
in $\ZFC$ it is equivalent to a first order statement: There is a $\ka$-complete ultrafilter over $\ka$,
that is, $\ka$ is measurable.
Recently Schlutzenberg \cite{S} showed that under the existence of proper class many 
weakly LS cardinals, the definition of critical cardinal  is equivalent to a certain first order sentence. Actually, under the assumption, he showed that
the ultrapower by certain extender sequence satisfies \L os' theorem,
so the critical cardinals can be formalized by the existence of suitable 
extenders.
Moreover his argument can formalize strong cardinal as well.
See \cite{S} for the proof of the following proposition.
\begin{prop}
Suppose there are proper class many weakly LS cardinals.
Then for every cardinal $\ka$ the following are equivalent:
\begin{enumerate}
\item For every $\alpha>\ka$,
there is a weakly LS cardinal $\la>\alpha$ of uncountable cofinality,
a transitive set $N$, and an elementary embedding $j:V_\la \to N$ such that
the critical point of $j$ is $\ka$, $\alpha<j(\ka)$,
and $V_\alpha \subseteq N$.
\item For every $\alpha$,
there is a transitive class $M$ and
an elementary embedding $j:V \to M$ such that
the critical point of $j$ is $\ka$, $\alpha<j(\ka)$, and $V_\alpha \subseteq M$.
\item For every $\alpha$,
there is a definable transitive class $M$ and
a definable elementary embedding $j:V \to M$ such that
the critical point of $j$ is $\ka$, $\alpha<j(\ka)$, and $V_\alpha \subseteq M$.
\end{enumerate}
\end{prop}
Other large cardinals defined by extenders, such as Woodin cardinal, superstrong cardinal, and $I_2$-embedding, can be formalized by a similar way.

\subsection*{Acknowledgments}
This research was supported by 
JSPS KAKENHI Grant Nos. 18K03403 and 18K03404.

%

\end{document}